\documentclass[twoside]{amsart}
\usepackage{amssymb,verbatim}
\usepackage{amsmath}
\usepackage{amsthm}
\usepackage{graphicx}
\usepackage{color}
\usepackage[all]{xy}
\newtheorem{theorem}{Theorem}[section]

\newtheorem{corollary}[theorem]{Corollary}

\newtheorem{Proposition}[theorem]{Proposition}
\newtheorem{Lemma}[theorem]{Lemma}

\theoremstyle{definition}
\newtheorem{Definition}[theorem]{Definition}
\newtheorem{example}[theorem]{Example}

\theoremstyle{remark}
\newtheorem{remark}[theorem]{Remark}

\newcommand{\cO}{{\mathcal O}}

\newcommand{\N}{\mathbb N}
\newcommand{\Q}{\mathbb Q}
\newcommand{\Z}{\mathbb Z}
\newcommand{\C}{\mathbb C}
\newcommand{\R}{\mathbb R}

\newcommand{\bb}[1]{\mathbb{#1}}
\newcommand{\m}[1]{\mathcal{#1}}

\newcommand{\dto}{\mathrm{\dashrightarrow}}

\newcommand{\s}{\sigma}

\newcommand{\f}{\varphi}
\newcommand{\ra}{\rightarrow}

\newcommand{\lin}{\sim}

\DeclareMathOperator{\Spec}{Spec}

\DeclareMathOperator{\Proj}{Proj}
\newcommand{\git}{\mathbin{
  \mathchoice{/\mkern-6mu/}% \displaystyle
    {/\mkern-6mu/}% \textstyle
    {/\mkern-5mu/}% \scriptstyle
    {/\mkern-5mu/}}}% \scriptscriptstyle

\begin{document}

%\today

\title[Adjunction]{Local Fano-Mori contractions of high nef-value} 
\author[Andreatta]{Marco Andreatta}
\author[Tasin]{Luca Tasin}

\thanks{We thank H. Ahmadinezhad, P. Cascini and M. Mella for helpful conversations. We are partially funded by the italian grants PRIN and GNSAGA-INdAM; the second author was funded by the Max Planck Institute for Mathematics in Bonn for part of the development of this project.} 

\address{Dipartimento di Matematica, Universit\'a di Trento, I-38050
Povo (TN)} 
\email{marco.andreatta@unitn.it}

\address{  Mathematical Institute of the University of Bonn, Endenicher Allee 60 D-53115
Bonn, Germany.} 
\email{tasin@math.uni-bonn.de}


\begin{abstract}  Let $X$ be a variety with terminal singularities of dimension $n$.

We study local contractions $f:X\ra Z$ supported by a $\Q$-Cartier divisor of the type  $K_X+ \tau L$, where $L$ is an $f$-ample Cartier divisor and  $\tau > 0$ is a rational number. Equivalently,  $f$ is a Fano-Mori contraction associated to an extremal face in $\overline {NE(X)}_{K_X+\tau L = 0}$. We prove that, if $\tau > (n-3) >0$, the general element $X' \in |L|$ is a variety with at most terminal singularities. We apply this to characterize, via an inductive argument, some birational contractions as above with $\tau > (n-3)\geq 0$.

\end{abstract}

\maketitle

\section{Introduction}

Let $X$ be a variety with at most log terminal singularities of dimension $n$; let $f:X\ra Z$ be a local contraction on $X$ (see Section \ref{notation}).
Assume that $f$ is an adjoint contraction supported by a $\Q$-Cartier divisor of the type  $K_X+ \tau L$, where $L$ is an $f$-ample Cartier divisor and $\tau$ is a positive rational number (Definition \ref{AC}). Equivalently,  $f$ is a Fano-Mori contraction associated to an extremal face in $\overline {NE(X)}_{K_X+\tau L = 0}$ (Definition \ref{FM} and Remark \ref{EQ}). These maps naturally arise in the context of the minimal model program. 

The description and the classification of such contractions $f:X \ra Z$ are often obtained by an inductive procedure, the so-called Apollonius method: it consists in finding a "good" element $X' \in |L|$ (that is an element of the linear system $|L|$ with good singularities), studying by induction the properties of $f_{|X'} :X' \ra Z'$ and then lifting  them to $f: X \ra Z$.
The first step, i.e. the proof of the existence of good elements in $|L|$, is a long lasting and delicate problem; the following is a result in this direction.  

\begin{theorem}  \label{n-3} Let $f:X\ra Z$, L and $\tau$ be as above; assume that $X$ has terminal singularities and $\tau > (n-3) >0$.
Let $X' \in |L|$ be  a general divisor.
Then $X'$ is a variety with at most terminal singularities and $f_{|X'} : X' \to f(X')=:Z'$ is a local contraction supported by 
$K_{X'}+ (\tau - 1) L'$, where $L' := L_{|X'}$  (i.e. $f'$ is again a Fano-Mori contraction). 
\end{theorem}

The next two results are proved by induction, applying Theorem \ref{n-3}.  If $n = 3$, then part A of the following Theorem is the main result of \cite{Kaw01}.

\begin{theorem} \label{birational} Let $f:X\ra Z$, $L$ and $\tau$ be as above; assume also that $X$ is terminal and $\Q$-factorial and that $\tau > (n-3)\geq 0$.

\begin{itemize} 

\item[A)] Assume that $f$ is  birational and contracts a prime divisor to a point. For $i=1, \ldots, n-3$, let $H_i \in |L|$ be a general divisor and set $X''= \cap H_i$.  Then $X''$ is a threefold with terminal singularities and $f'': X'' \to Z''$ is a divisorial contraction of an irreducible $\Q$-Cartier divisor $E'' \subset X''$ to a point $p \in Z''$. Assume that $p$ is smooth in $Z''$.    
Then  $f$ is a weighted blow-up of a smooth point with weight $(1,a,b,c,\ldots,c)$, where $a,b$ are positive integers, $(a,b)=1$, $c$ is the positive integer such that  $L= f^*f_*L - cE$ and $ab | c$.

\item[B)] Let $E$ be the exceptional locus of $f$. Assume that  $X$ has only  points of index $1$ and $2$ and that each component of $E$ has dimension  $(n-2)$ (in particular $f$ is a birational small contraction). Then  $\tau= \frac{2n-5}{2}$, $E$ is irreducible, it is contracted to a point and $(E,L_{|E})=(\bb P^{n-2}, \m O(1))$. 

\end{itemize}
\end{theorem}

Fano-Mori contractions of nef-value $\tau > (n-2)$ are classified, see \cite{And13} and \cite{AT14}. In \cite{AT14} we also describe  divisorial contractions of nef-value $\tau > (n-3)$ such that the exceptional locus is not contracted to a point . 
The above Theorem is a further step towards a classification in the case $(n-2) \geq \tau > (n-3)$.

\section{Notation}
\label{notation}

We use notations and definitions which are standard in the Minimal Model Program, they are compatible with the ones in the books  \cite{KollarMori} and \cite{Laz04}.

In particular a \emph{log pair} $(X,D)$ consists of a normal variety $X$ together with an effective Weil $\Q$-divisor $D=\sum d_i D_i$ on $X$ such that $K_X+D$ is $\Q$-Cartier.

Let $\mu: Y \to X$ be a log resolution of $(X,D)$, then we can write
$$
K_Y + \mu^{-1}_* D = \mu^*(K_X+D) + \sum_{E_i \ \mbox{\scriptsize{exceptional}}} a(E_i,X,D)E_i.
$$

We define the \emph{discrepancy} of $(X,D)$ as 
$$\rm{discrep}(X,D):= \inf_E \{a(E,X,D) : E\mbox{ is an exceptional divisor over } X \}.$$

We say that $(X,D)$ is terminal, resp.canonical, klt (or Kawamata log terminal), plt,  lc   (or  log canonical) if discrep$(X,D)$  is $>0$, resp. 
$\geq 0$, $ >-1$  and $ \lfloor D \rfloor =0$, $>-1$, $\geq -1$.

If $D=0$, then the notions  klt and plt coincide and $X$ is  called log terminal (lt).

The \emph{log canonical threshold} of a log pair $(X,D)$ is defined as
$$
\textrm{lct}(X,D):=\sup\{t\in\Q : (X,tD)\mbox{ is log canonical}\}.
$$
\label{lc}

A subvariety $W \subset X$ is called a \emph{lc centre} for $(X,D)$ if there is a log resolution $\mu: Y \to X$ and an irreducible exceptional divisor $E$ on $Y$ such that $a(E,X,D)=-1$ and $\mu(E)=W$. 
The set of all the lc centres is denoted by $CLC(X,D)$. Note that if $W_1, W_2 \in CLC(X,D)$ and $W$ is an irreducible component of $W_1 \cap W_2$, then $W \in  CLC(X,D)$; in particular, there exist minimal elements in  $CLC(X,D)$.  An lc centre $W$ is called \emph{isolated} if for any  log resolution $\mu: Y \to X$ and any  exceptional divisor $E$ on $Y$ such that $a(E,X,D)=-1$, we have $\mu(E)=W$.

\bigskip
Let $T$ be a normal projective variety over $\C$ and $n = \dim T$. A  \emph{contraction} is a surjective morphism  $\f: T \ra S$ with connected fibres onto a normal variety $S$.  We take a contraction $\f: T \ra S$  and we fix a non trivial fibre $F$ of $f$; take an open affine set $Z \subset S$ such that $f(F) \in Z$. 

Let $X := f^{-1}(Z)$; then $f:X\ra Z$ will be called a \emph{local contraction around $F$}, or simply a local contraction; eventually shrinking $Z$, we can assume that $\dim F \geq \dim F'$ for every fibre $F'$ of $f$.

We assume that $f$ is \emph{projective}, that is we assume the existence of $f$-ample Cartier divisors $L$. We will also assume that $X$ has log terminal, or milder type, singularities.

\begin{Definition}  \label{FM}
We will say that a  local projective contraction $f:X \ra Z$ is \emph{Fano-Mori} (F-M) if  $-K_X$ is $f$-ample.  
\end{Definition}

Fano-Mori contractions are associated to extremal faces of the polyhedral part of the Mori-Kleiman cone  $\overline {NE(X)}_{K_X <0}= \{[C] \in  \overline {NE(X)}: K_X.C < 0\}$  in the vector space $N_1(X)$ generated by 1-cycles modulo numerical equivalence. In particular the contraction contracts exactly all the curves contained in the associated face. If the associated face has dimension $1$ (a ray) the contraction is called \emph{elementary}.

\begin{Definition} \label{AC}
We will say that a  local projective contraction $f: X \ra Z$ is an \emph{adjoint contraction supported by $K_X+\tau L$} if there is a $\tau \in \Q$ such that $K_X+ \tau L \lin_f \cO_X$,  where $L$ is an $f$-ample Cartier divisor ($\lin _f$ stays for numerical equivalence over $f$).
\end{Definition}  

\begin{remark}  \label{EQ}
Any F-M contraction $f:X \ra Z$, once we fix a  $f$-ample Cartier divisors $L$, is an adjoint contraction. To see this we define the 
{\it nef-value} of the pair $(f:X \ra Z,L)$ as 
$\tau_f(X,L):=\hbox{inf} \{t \in \R : K_X + tL \hbox{ is $f$-nef}\}.$
By the rationality theorem of Kawamata (Theorem 3.5 in  \cite{KollarMori}), $\tau(X,L)$ is a rational non-negative number and therefore $f$ is an adjoint contraction supported by  $K_X+ \tau L$.
Viceversa any adjoint contraction with positive $\tau$ is clearly a F-M contraction.
\end{remark}

\medskip
All through the paper, although not further specified, we will be in the following set up:
\begin{enumerate} 
\item[$(\star)$] $X$ is a variety with at most log terminal singularities,  $f: X \to Z$ is an adjoint contraction (Definition \ref{AC}), local around a (non trivial) fibre $F$ and supported by $K_X + \tau L$, where $L$ is an $f$-ample Cartier divisor and $\tau$ is a rational number.
\end{enumerate}  

We will denote by $E$ the exceptional locus of $f$ and by $Bs|L|$  the relative base locus of $L$, i.e. the support of the cokernel of the natural map $f^* f_* L \to L$. Clearly $Bs|L| \subset E$.

\bigskip
Weighted projective spaces and weighted blow-up, under some conditions on the weights, are special Fano-Mori contractions.
For a  detailed treatment of  weighted blow-ups we refer to Section 10 of \cite{KollarMori} or Section 3 of \cite{AT14}; here we just fix our notation.

Let $\sigma=(a_1,\ldots,a_n) \in \mathbb N^n$ such that $a_i >0$ and $\gcd(a_1,\ldots,a_n)=1$. 

We denote by  $ \bb P(a_1, \ldots,a_n)$ the \textit{weighted projective space} with weight $(a_1,\ldots,a_k)$.

Let $X=\mathbb A^n = \Spec  \C[x_1, \dots, x_n]$  and $p=(0,\ldots,0) \subset X$.  
Consider the rational map
$
\f: \bb A^n \to \bb P(a_1, \ldots,a_n)
$
given by $(x_1,\ldots,x_n) \mapsto (x_1^{a_1}:\ldots:x_n^{a_n})$.
The \textit{weighted blow-up} of $p \in X$ of weight $\sigma$  is defined as the closure $\overline X$ in 
$\bb A^n \times \bb P(a_1, \ldots,a_n)$ of the graph of $\f$, together with the morphism $\pi: \overline X \to X$ given by the projection on the first factor.
The map $\pi$ is birational and contracts an exceptional irreducible divisor $E\cong \mathbb P(a_1,\ldots,a_n)$ to $p$. 
For any $d \in \N$ we define the $\sigma$-weighted ideal of degree $d$ as 
$ 
I_{\s,d}:= (x_1^{s_1}\cdots x_n^{s_n} : \sum_{j=1}^n s_ja_j \ge d).\\
$
We have the following characterization:
$
\overline X= \Proj (\bigoplus_{d \ge 0} I_{\sigma,d})$ 
(see  \cite{AT14}).

A criterium to check that the singularities of $\overline X$ are terminal can be find in \cite[Theorem 4.11]{Re87}: for instance if  
 $\sigma=(1,a,b,c, \ldots,c)$, where $(a,b)=1$ and $ab |c$, then $\overline X$ has terminal singularities.

\section{Existence of good sections}\label{goodsections}

In this section we prove Theorem \ref{n-3} and we provide a collection of technical results which could be useful by themselves
(see Proposition \ref{Bs}).  

\medskip
We start with a non-vanishing lemma.

\begin{Lemma} \label{non-vanishing} Let $f: X \to Z$ be as in  Section \ref{notation}  $(\star)$. Let $D\lin_f \beta L$ be a $\Q$-divisor such that  $(X,D)$ is lc and let $W\in CLC(X,D)$ be a minimal centre.  Assume that $\tau - \beta > -1$, or that  $\tau - \beta \geq -1$ if $f$ is birational; assume also that one of the following conditions is satisfied:
\begin{itemize}
\item[(i)] $\dim W\leq 2$, 
\item[(ii)] $\dim W \ge 3$ and $\tau-\beta > \dim W - 3$. 
\end{itemize}
Then $H^0(W,L_{|W}) \ne 0$.
\end{Lemma}

\begin{proof}
By subadjunction formula (see Theorem 1.2 of \cite{FG12}), there is an effective $\Q$-divisor $D_W$ such that $(W, D_W)$ is klt and 
$$
K_W+ D_W \lin (K_X+D)_{|W}\lin-(\tau-\beta)L_{|W}.
$$  

If $\dim W \le 2$, then we conclude by Theorem 3.1 of \cite{Kaw00}.

If $\dim W \ge 3$, then  $(W,D_W)$ is a log Fano variety of index $i(W,D_W) > \dim W - 3$ and  the result follows by the main Theorem of \cite{Am99}.

\end{proof}

The next is the first step to prove the existence of a good element in the linear system $|L|$. 

\begin{corollary}  Let $f: X \to Z$ be a as in Section \ref{notation}  $(\star)$. Let $D\lin_f \beta L$ be a $\Q$-divisor such that  $(X,D)$ is lc and let $W\in CLC(X,D)$ be a minimal centre.  Assume that $\tau - \beta > -1$ or that  $\tau - \beta \geq -1$ if $f$ is birational; assume also that one of the following conditions is satisfied:
\begin{itemize}
\item[(i)] $\dim W\leq 2$, 
\item[(ii)] $\dim W \ge 3$ and $\tau-\beta > \dim W - 3$. 
\end{itemize}
Then there exists a section of $|L|$ not vanishing identically on $W$.
\label{lccentre}
\end{corollary}

\begin{proof}
By a tie-breaking technique (see the discussion 1.15 in \cite{Mel992}), we may assume that $W$ is an isolated lc centre and hence $I_W= \m J(D)$, where $I_W$ is the ideal sheaf of $W$ and $J(D)$ is the multiplier ideal of $D$ (see Lemma 2.19 of \cite{CKL11}). Consider the exact sequence
$$
0 \to \m O_X(L) \otimes \m I_W \to \m O_X(L) \to \m O_W (L_{|W}) \to 0.
$$
Since $L-(K_X+ D) \lin_f (1+\tau-\beta)L$ is $f$-nef and big, we can apply Nadel vanishing \cite[Thm. 9.4.17]{Laz04}  to obtain that
$$
 H^0(X, L) \to H^0(W, L_{|W})
$$
 is surjective.
The result follows now by Lemma \ref{non-vanishing}.

\end{proof}

The next proposition collects a series of useful technical results. 

\begin{Proposition}\label{Bs}
Let $f: X \to Z$ be as in Section \ref{notation}  $(\star)$.

\begin{itemize}

\item[1] (\cite[Theorem 5.1]{AW93})  Assume that either $\dim F < \tau+1$, if $f$ is of fibre type, or $\dim F \le \tau+1$, if $f$ is birational.
Then $L$ is relatively base-point free (i.e. $Bs|L| = \emptyset$)\label{AW}.

\item[2] If $\tau > -1$ and $\dim F < \tau +3$, then there exists a section of $|L|$ not vanishing identically along $F$\label{section}.

\item [3] Assume that $\dim F < \tau+3$,  $F$ is irreducible,  and that either $\tau > 0$, if $f$ is of fibre type, or $\tau \ge 0$, if $f$ is birational.
Then the general element of $|L|$ is a variety with lt singularities. 
If $\dim F < \tau +2$, then the same holds without the assumption that $F$ is irreducible. 

\item[4] \label{LT} Assume $\tau >0$ and  $n-3 < \tau$. Then $\dim Bs|L| \le 1$.

\item[5]  \label{ter} Assume  $\dim F < \tau+3$, $F$  irreducible and $\tau \ge 1$. Let $S \in |L|$ be a general element.  If $X$ has canonical  singularities, then $S$ has canonical  singularities.    If $X$ has terminal  singularities, then $S$ has terminal  singularities, except possibly when $\tau =1$ and $f$ is of fibre-type. \\
If $\dim F < \tau +2$, then the same holds without the assumption that $F$ is irreducible.

\item[6] \label{gorenstein1} Assume that $\dim F < \tau+3$, $F$ is irreducible and  $\tau > 0$ if $f$ is of fibre type or $\tau \ge 0$ if $f$ is birational. If $X$ has canonical Gorenstein singularities, then the general element of $|L|$ has canonical singularities. 

\item[7] \label{gorenstein2} Assume that $\dim F = \tau+3$, $F$ is irreducible and  $\tau > 0$ if $f$ is of fibre type or $\tau \ge 0$ if $f$ is birational. If there exists a section of $L$ not vanishing along $F$  and $X$ has canonical Gorenstein singularities, then the general element of $|L|$ has canonical singularities. 

\end{itemize}

\end{Proposition}

\begin{remark}
Point 1 is the main result of \cite{AW93}.  Points 2 and 3 are generalisations of Proposition 2.4 and Proposition 3.3 in \cite{Mel992}. Points 4, 5 and 6 are generalizations of results in \cite{Mel99} and \cite{Mel992}. Point 7 is the analogous of \cite[Thm. 1.1]{Flo13} in the relative set-up. 

At the Points 3 and 6 of Proposition \ref{Bs} the assumption $\tau > 0$ if $f$ is of fiber type is necessary, as the following trivial example shows. Let $E$ be a smooth elliptic curve and $D$ an ample line bundle with a base point (i.e. $D = {p})$.
Consider $X=E \times \bb P^m \to \bb P^m$ for $m \ge 0$ and $L=D \boxtimes (-2K_{\bb P^m})$. This is  an adjoint contraction of  fibre-type with $\tau=0$ for which the conclusions of Points 3 and 6 do not hold. Similar examples can be constructed for point 7.

Counter-examples for the statement in the point 5 for $\tau =1$ and $f$ of fiber type were given by Mella; in  \cite{Mel99} he actually classified all terminal Mukai 3-folds $Y$ such that the general element of $|-K_Y|$ is not smooth. Taking $X:=Y \times \bb P^m \to \bb P^m$ for $m \ge 0$ and $L=-(K_Y \boxtimes 2K_{\bb P^m})$, we get examples of fibre-type contractions (not necessarily to a point) with $\tau=1$ which do not satisfy the conclusions of Point 5.
\end{remark}

\medskip
\begin{proof}[Proof of Proposition \ref{section}.2]
Let $\{h_i\} \in H^0(Z, \m O_Z)$ be general functions  vanishing at $f(F)$  such that $(X,D)$ is not lc, where $D=\sum  f^*(h_i)$.
Let $\gamma= \mathrm{lct}(X,D)$ and let $W \in CLC(X, \gamma D)$ be a minimal lc centre; by the general choice of $h_i$ outside $f(F)$, we can assume that $W  \subset F$. Note that $\gamma D \lin_f 0$ and that, by assumption, $\dim W \leq \dim F <\tau +3$. Therefore by Corollary \ref{lccentre} there exists a section of $|L|$ not vanishing identically on $W$ and thus on $F$.

\end{proof}

\begin{proof}[Proof of Proposition \ref{LT}.3]

We start proving that $Bs|L|$ has codimension at least two. Assume by contradiction that there exists an irreducible component $V \subset Bs|L|$  of dimension $n-1$.	

Suppose first that $V\subset F$. Let $H \in |L|$ be a general element and set $c=\mathrm{lct}(X,H)$. If $c <1$, then $LCC(X,cH) \subset Bs |L|$ ;  consider a minimal lc centre $W \in CLC(X,c H)$.   By  Proposition \ref{section}.2, $W \subsetneq F$. If $F$ is irreducible, then $\dim W \leq \dim F -1 < \tau+2$. If $F$ is not irreducible, then $\dim W \le \dim F < \tau +2$ by hypothesis. Therefore by Corollary \ref{lccentre} there exists a section of $|L|$ not vanishing identically on $W$,  thus on $Bs|L|$, which is a contradiction. If $c=1$, then $V \subset Bs|L|$ is an lc centre of $(X,H)$ and, by  Proposition \ref{section}.2, $V \subsetneq F$. Since $dim V = n-1$, 
$f$ is a contraction to a point. Therefore, by assumptions, we have $\tau >0$. We can conclude again by Corollary \ref{lccentre}.

%We can reduce to the case $Bs|L|\subset F$, via a so called a vertical slicing argument. 
Assume now that $V$ is not contained in any fibre of $f$ and consider $h_1, \ldots, h_d$ general functions on $Z$, where $d := \dim f(V) >0$. Set $X_{h_i}= f^* h_i$ and $X'=\cap X_{h_i}$. Note that $\dim X'= n-d$. By vertical slicing (\cite[Lemma 2.5]{AW93}), we get a local contraction $f': X' \to Z'$, supported by $K_{X'}+ \tau L'$ where $L'=L_{|X'}$ and there exists an irreducible component $V'$ of $V \cap X' \subset Bs|L'|$ (actually, by Bertini, $V \cap X'$ is  irreducible if it has positive dimension) such that $\dim V'=n-d-1$ and $V' \subset F'$, where $F'$ is a fibre of $f'$. Note that if $f'$ is of fiber type also $f$ is of fiber type, therefore in this case $\tau$ is positive  by assumption.  We are in the situation of the previous step and we can reach a contradiction.

\medskip

We now prove that the general element of $|L|$ has lt singularities. Let $S \in |L|$ be general element; by Bertini Theorem (see \cite[Thm. 6.3]{Jou83}) and the fact that $Bs|L|$ has codimension at least two, we see that $S$ is irreducible and generically reduced. Assume by contradiction that $S$ has singularities worse than log terminal. Then, by Proposition 7.5.1 of \cite{Kol97}, $(X,S)$ is not plt. 

Assume  first that  $\tau >0$. Set $\gamma=\mathrm{lct}(X,S) \le 1$ and consider a minimal lc centre $W \in CLC(X, \gamma S)$ such that $W \subset Bs|L|$ (such a center exists by Bertini Theorem, see for instance \cite[Lemma 5.1]{Am99}).   We want to show that there is a section of $|L|$ not vanishing identically on $W$, obtaining in this way a  contradiction.

As above, via a vertical slicing argument, we may assume $W \subset F$. In fact, let $d= \dim f(W)$.  Consider $h_1, \ldots, h_d$ general functions on $Z$. Set $X_{h_i}= f^* h_i$ and $X'=\cap X_{h_i}$.  By vertical slicing (\cite[Lemma 2.5]{AW93}), we get a local contraction $f': X' \to Z'$ around a fibre $F'$, supported by $K_{X'}+ \tau L'$ where $L'=L_{|X'}$. Let $S' \in |L'|$ be general. Since each $X_{h_i}$ is general and intersects $W$,  we have that $LLC(X', \gamma S') \subset W \cap X' \subset F'$  and the claim is proved.  

 By  Proposition \ref{section}.2, $W \subsetneq F$. If $F$ is irreducible, then $\dim W \leq \dim F -1 < \tau+2$. If $F$ is not irreducible, then $\dim W \le \dim F < \tau +2$ by hypothesis.   If $\dim W \ge 3$, then $\tau - \gamma > \dim W -3 \ge 0$ and we can apply point (ii) of Corollary \ref{lccentre}. If $\dim W \le 2$, then the contradiction follows  by point (i) of Corollary \ref{lccentre}.

Assume now that  $\tau =0$ and $f$ is not of fibre-type. Let $H= \varepsilon f^*(h)$, where $h$ is a general function on $Z$ vanishing at $f(F)$ and $0<\varepsilon << 1$.  Set $D= S + H$ and $\delta= \mathrm{lct}(X,D) < 1$. We can consider a minimal centre $W \in CLC(X, \delta D)$ and reason as before.

\end{proof}

\begin{proof}[Proof of Proposition \ref{Bs}.4]

If $\dim F \leq (n-2)$ then  \ref{Bs}.4 follows from the main Theorem of \cite{AW93}, as quoted in \ref{Bs}.1.
Assume  that $F \geq (n-1)$, then the result follows by the next Lemma.
\end{proof}

\begin{Lemma}\label{divisorial}
Assume that $X$ has log terminal singularities, $\tau >0$ and  $\dim F=n-1 < \tau +2$. Then $\dim Bs|L| \le 1$.
\end{Lemma}

\begin{proof}

The proof of the Lemma is by induction on $n \ge 3$. We have proved above that $|L|$ has not fixed components, therefore the lemma is true for $n \le 3$.

Assume $n >3$. Let  $X' \in |L|$ general. Since $|L|$ has no fixed component, by Bertini we get that $X'$ does not contain any irreducible component of $F$ (and that it is irreducible and reduced). Moreover, by Proposition \ref{LT}.3, we have that $X'$ is log  terminal. Hence, by horizontal slicing (\cite[Lemma 2.6]{AW93}), $f: X' \to Z'$ is a contraction supported by $K_{X'} + (\tau-1)L_{|X'}$ around a fibre $F'= F \cap X'$.  It also follows that $\dim Bs|L| \le \dim Bs |L'|$, because any section of $L'$ lifts to a section of $L$ by \cite[Lemma 2.6.1]{AW93}. By induction, we are done.

\end{proof}

\begin{proof}[Proof of Proposition \ref{ter}.5]
Let $S$ be a general element of $|L|$; by Proposition \ref{LT}.3,  $S$ has lt singularities.
Let  $\mu:Y\ra X$ be a log resolution of the pair $(X,S)$ and of the base locus of $|L|$.
We can write
 
$$\mu^*S = \overline S +\sum_i r_iE_i$$

$$K_Y=\mu^*K_X+\sum_i a_iE_i$$

$$K_Y+ \overline S =\mu^*(K_X+S)+\sum_i (a_i- r_i)E_i$$

where $ \overline S = \mu_*^{-1}S$ is the strict transform of $S$ and $|\overline S|$ is basepoint free. Moreover, $r_i \in \bb N$ and $r_i \ne 0$ if and only if $\mu(E_i) \subset Bs|L|$.

Assume that $S$ has not canonical singularities (resp. terminal singularities); after reordering we can assume that $a_0 < r_0$ (resp. $a_0 \le  r_0$). 
Since $S$ is generic, by Bertini we can assume that $\mu(E_i) \subset Bsl|L|$, for all $i$ such that $r_i >0$.

Let $D= S + S_1$, where $S_1$ is another generic section in $|L|$; note that $\mu$ is a log resolution also  for
the pair $(X,D)$. Let $r_0^1 \geq 1$ be the multiplicity of $S_1$ at the centre of valuation associated to $E_0$. Then $(X,D)$ is not LC since $a_0 +1 < r_0 + r_0^1$ (resp. $a_0 +1 \le r_0 + r_0^1$). Let $\gamma=\mathrm{lct}(X,D) \le 1$ and $W \in CLC(X,\gamma D)$ be a minimal lc centre. Now we can reason as in the proof of  Proposition \ref{LT}.3. 	
\end{proof}

\begin{proof}[Proof of Proposition \ref{gorenstein1}.6]
In the notation of the proof of Proposition \ref{ter}, assume by contradiction that $S$ is not canonical. Then $a_i - r_i <0$ for some $i$; since $a_i$ and $r_i$ are integers, we get $a_i-r_i \le -1$ and hence  $(X,S)$ is not plt. Set  $\gamma= \mathrm{lct}(X,S) \le 1$ and let $W \in CLC(X,\gamma S)$ be minimal lc centre. Now, as in the proof above, we derive a contradiction.
\end{proof}

\begin{proof}[Proof of Proposition \ref{gorenstein2}.7]
If $f$ is a contraction to a point, then the result is exactly  \cite[Thm. 1.1]{Flo13}, so assume that $f$ is not a contraction to a point. Let $S \in |L|$ be general and assume by contradiction that $S$ is not canonical. Then $(X,S)$ is not plt.  Let $H= \varepsilon f^*(h)$, where $h$ is a general function on $Z$ vanishing at $f(F)$ and $0<\varepsilon << 1$.  Set $D= S + H$ and $\delta= \mathrm{lct}(X,D) < 1$. We can consider a minimal centre $W \in CLC(X, \delta D)$ and reason as in the proof above.
\end{proof}

\bigskip

\begin{proof}[Proof of Theorem \ref{n-3}]
 
The fact that $X'$ is terminal follows by Proposition \ref{ter}.5.
The fact that  $f_{|X'} : X' \to Z'$ is a local contraction supported by 
$K_{X'}+ (\tau - 1) L'$ follows by the so called horizontal slicing (\cite[Lemma 2.6]{AW93}).
\end{proof}

\section{Lifting of contractions}\label{s_lifting}

Let $X$ be a terminal variety of dimension $n \ge 4$ and let $f:X \to Z$ be a local contraction supported by $K_X+\tau L$ such that 
$\tau > n-3$; assume that $f$ contracts a prime $\Q$-Cartier divisor $E$ to a smooth point $p\in Z$. 

By Theorem \ref{n-3} the general $X' \in |L|$ has terminal singularities and  $f'=f_{|X}: X' \to Z'$ is a divisorial contraction to $p \in Z'$.    
Since $f_* L$ is a Cartier divisor let $c$ be a positive integer $c$ such that $f^*f_* L= L+cE$. 

\begin{Lemma}\label{lift}
In the situation above, assume that $p$ is smooth in $Z'$ and that $f'$ is a weighted blow-up of type $(1,a,b,c\ldots,c)$, where $c$ appears $(n-4)$ times. Then $f$ is a also a weighted blow-up of type $(1,a,b,c,\ldots,c)$, where $c$ appears $(n-3)$ times.
\end{Lemma}

\begin{proof}
Let $x_1,\ldots,x_n$  local coordinates for $p$; we may also assume that  $f_*(X')=\{x_n=0\}$.

Note that $\m O_X(-cE)$ is  $f$-ample  and that the map $f$ is proper; so we have that %by (8.8.1.3) of \cite{EGA II}, we know that
$$
X= \Proj ( \oplus_{d \ge 0}  f_*\m O_X(-dcE)).
$$ 

Using the notation of Section \ref{notation}, we need to prove that
$$
f_*\m O_X(-dcE)=(x_1^{s_1}\cdots x_n^{s_n}  :  s_1 + s_2a + s_3 b+ \sum_{j=4}^n c s_j \ge dc).
$$ 

The proof is by induction on $d \ge 0$.

Consider the exact sequence
$$
0 \to \m O_X(-L-dcE) \to \m O_{X} (-dcE)\to \m O_{X'}(-dcE) \to 0.
$$

Note that
$$
-L-dcE \sim_f -(d-1)cE \sim_f K_X + (n-3 + d-1+ \frac{a+b}{c})L, 
$$

Hence, pushing down to $Z$ the above exact sequence and applying the relative Kawamata-Viehweg Vanishing, we have

\begin{align}\label{sequence3}
0 \to  f_* \m O_X(-(d-1)cE) \stackrel{\cdot x_n}\rightarrow f_*\m O_{X} (-dcE)\to f_*\m O_{X'}(-dcE) \to 0.
\end{align}

Since by assumption $f'$ is a weighted blow of type $(1,a,b,c,\ldots,c)$, we have
$$
f_*\m O_{X'}(-dcE)=(x_1^{s_1}\cdots x_{n-1}^{s_{n-1}}  :  s_1 + s_2a+s_3b + \sum_{j=4}^{n-1}cs_j \ge dc),
$$
where $s_j \in \bb N$.  By induction on $d$, we can also assume that
$$
f_*\m O_{X}(-(d-1)cE)=(x_1^{s_1}\cdots x_n^{s_n}  :  s_1 + s_2a +s_3b + \sum_{j=4}^n cs_j \ge (d-1)c),
$$
the case $d=0$ being trivial.

Let $g=x_1^{s_1}\cdots x_n^{s_n} \in f_*\m O_{X}(-dcE)$ be a monomial. If $s_n \ge 1$ then $g$, looking at the sequence \eqref{sequence3}, comes from $f_*\m O_{X}(-(d-1)cE)$ by the multiplication by $x_{n}$; therefore 
$$
s_1 + s_2a +s_3b+ \sum_{j=4}^{n-1} s_j c +s_n c \ge  (d-1)c + s_n c \ge dc.
$$

If $s_n=0$, then  $g \in f_*\m O_{X'}(-dcE)$ and so 
$$
s_1 + s_2a +s_3b+ \sum_{j=4}^{n} s_j c= s_1 + s_2a + s_3b+\sum_{j=4}^{n-1} s_j c \ge dc.
$$

The non-monomial case follows immediately.
\end{proof}

\bigskip
\begin{proof}[Proof of Theorem \ref{birational}.A]

Let $H_i \in |L|$ be general divisors for $i=1,\ldots, n-3$.
By Theorem \ref{n-3}, for any $i$,  $H_i$ is a variety with terminal singularities and the morphism $f_i=f_{|H_i}: H_i \to f(H_i)=:Z_i$ is a local contraction supported by $K_{H_i}+ (\tau -1 )L_{|H_i}$.
Since $Z$ is terminal and $\Q$-factorial (see \cite[Corollary 3.36]{KollarMori} and \cite[Corollary 3.43]{KollarMori}), then  the $Z_i$'s are  $\Q$-Cartier divisors on $Z$. 
 
For any $t=0,\ldots,n-3$ define $Y_t=\cap_{i=1}^{n-3-t}  H_i$  and $g_t= f_{|Y_t}: Y_t \to g_t(Y_t)=:W_t$; in particular  $Y_{n-3}=X$, $g_{n-3}=f$ and $W_{n-3}=Z$. Let, as in the statement of the Theorem, $X''= Y_0$ and $f''= g_{0}$.

By induction on $t$, applying Theorem \ref{n-3},  one sees that, for any $t=0,\ldots,n-4$, $Y_t$ is terminal and $g_t= f_{|Y_t}: Y_t \to W_t$ is a Fano Mori contraction. Therefore  $W_t$ is a terminal variety (by \cite[Corollary 3.43]{KollarMori}) and it is a $\Q$-Cartier divisor in $W_{t+1}$, because intersection of $\Q$-Cartier divisors (by construction $W_t=\cap_{i=1}^{n-3-t} Z_i$). 
Therefore by \cite[Lemma 1.7]{Mel97}, and by induction on $t$, it follows that $p$ is a smooth point in  $W_t$, for all $t$.
  
Set $L_t:= L_{|W_t}$. Since $Bs|L_t|$ has dimension at most 1 by Proposition \ref{Bs}.4, by Bertini's theorem (see \cite[Thm. 6.3]{Jou83}) $E_t:= Y_t \cap E$ is a prime divisor.  $E_t$  is the intersection of $\Q$-Cartier divisors and hence it is $\Q$-Cartier.  

Therefore $f'': X'' \to Z''$ is a divisorial contraction from a 3-fold $X''$ with terminal singularities, which contracts a prime $\Q$-Cartier divisor $E'':= E_0$ to a point $p\in Z''$, which we assume to be smooth. By \cite{Kaw01} we know then that $f''$ is a blow-up of type $(1,a,b)$ (note that in  \cite{Kaw01}  the $\Q$-factoriality of the domain is not needed, see also \cite[Thm. 1.9]{Kaw03}).

We conclude by induction on $t$ applying Lemma \ref{lift}.  

\end{proof}

\bigskip
\begin{proof}[Proof of Theorem \ref{birational}.B]

We first show that $E$ is contracted to a point. By \cite[Theorem 2.1]{And95} $dimf(E) \leq 1$. Since $\dim E=n-2$ and the non-Gorenstein locus of $X$ has codimension $3$, if $dimf(E) = 1$ then there is a fiber which is not contained in the non-Gorenstein locus; by \cite[Lemma2.1]{BHN13} we get a contradiction. (See the following Remark \ref{Small} for a further analysis).

 By the rationality theorem, \cite[Theorem 4.1.1]{KMM85}, we have $ 2\tau = \frac{u}{v}$ where $u,v \in \N$ and $u \le {2(n-1)}$. Therefore we have :
$$
n-3< \tau=\frac{u}{2v}\le \frac{n-1}{v}.
$$

If $n=4$ this gives $v=1$ and $u=3$ or $v= 2$ and $u=5$. 
If $n >4$  we can have only $v=1$ and $u=2n-5$.

 We want to exclude the case $n=4$ and $\tau=5/4$. Assume by contradiction that $4K_X+5L$ is a supporting divisor for $f$ and set $H=2K_X+3L$. Then $H$ is an ample Cartier divisor such that
$$
2K_X+5H=3(4K_X + 5L).
$$
This implies that $2K_X+5H$ is also a supporting divisor for $f$ and that $5/2=\tau(X,H)$, which is impossible because in dimension 4 birational contractions with nef-value greater than 2 are divisorial (see \cite{AT14}).

\medskip
By \cite[Theorem 5.1]{AW93}  we can suppose that $L$ is globally generated. Pick $(n-3)$ general members $H_i \in |L|$ ($1\le i \le n-3$) and let $X'= \cap H_i$ be the scheme intersection. By Theorem \ref{n-3} $X'$ is a $3$-fold with terminal singularities and, by horizontal slicing (\cite[Lemma 2.6]{AW93}), the restricted morphism $f':= f_{|X'}: X' \to Z'$ is a small contraction supported by $K_{X'}+(\tau-n+3)L_{|X'}$ with exceptional locus $C= (\cap H_i) \cap E$.
Note also that $X'$ has terminal singularities and has index at most 2, in fact $2K_{X'}=2(K_X + (n-3)L)_{|X'}$ is Cartier.  

Small contractions on a $3$-fold with terminal $2$-factorial singularities are classified in \cite[Theorem 4.2]{KM92}.  In particular this gives that $C$ is irreducible and isomorphic to $ \bb P^1$ and $-K_{X'}.C=\frac{1}{2}$. 

Therefore also $E$ is irreducible. Moreover, $\tau=\frac{2n-5}{2}$ implies  $L_{|X'}.C=1$ and thus $L_{|E}^{n-2}=1$.
%if $n=4$ and $\tau=\frac{5}{4}$, then $L_{|E}^{n-2}=2$.

By \cite[Thm. 2.1]{And95} we have that $E$ is normal and $\Delta(E,L)= 0$; by the classification of  varieties with  $\Delta$-genus equal to zero, we get that $(E,L)=(\bb P^{n-2}, \m O(1))$. %or $(E,L)=(\m Q^{n-2}, \m O(1))$.

\end{proof}

\begin{example} We construct a family of  examples of small contractions as in Theorem \ref{birational}.B.  We follow a construction via GIT as explained in \cite{Re92} and further in \cite{Br99}. Our examples are just higher dimensional versions of the examples of point $(4)$ of the main theorem in \cite{Br99}, to which we refer for more details.

Fix $n\ge 3$. Let $x_1,\ldots,x_{n-1},y_1,y_2,z$ be coordinates on $\C^{n+2}$ and consider the diagonal action of $\C^*$ on $\C^{n+2}$ with weights $(1,2,\ldots,2,-1,-1,0)$, that is for any $\lambda \in \C^*$ we have  $x_1 \mapsto \lambda x_1$, $x_i \mapsto \lambda^2 x_i$ for $i=2,\ldots,n-1$, $y_j \mapsto \lambda^{-1} y_j$ for $j=1,2$ and $z \mapsto z$.

Let 
$$
f=x_1y_1 + (x_2+\ldots + x_{n-1})y_2^2 + z^k
$$ 
with $k \ge 0$ and consider the hypersurface $A: \{f=0\} \subset \C^{n+2}$.  In the notation of \cite{Br99}, we are considering an action of type $(1,2,\ldots,2,-1,-1,0;0)$.  

Setting $B^-=A \cap \{x_1=\ldots=x_{n-1}=0 \}$ and $B^+=A \cap \{y_1=y_2=0 \}$ we can define $X= A \git \C^*$, $X^-= A^- \git \C^*$ and $X^+ = A^+ \git \C^*$ to  obtain the diagram 

\[
\xymatrix{X^- \ar @{-->}[rr]  \ar[rd]_{f^-} && X^+ \ar[ld]^{f^+} \\ & X}
\]

It is not difficult to check that this construction gives  a flip $X^- \dto X^+$ with exceptional loci $E^-=\bb P(1,2,\ldots,2) \cong \bb P^{n-2}$ and $E^+=\bb P^1$. Since $K_{X^-} \sim \m O(2n-5)$ we obtain that the contraction $f^-$ is supported by $2K_{X^-}+(2n-5)L$, where $L=\m O(2)$.  Finally, note that the singular locus of $X^+$ is of the form $ \C^{n-3} \times P $ where
$$
P=0 \in (x_1y_1 + y_2^2 + z^k)/\Z_2(1,1,1,0)
$$
is a $cA/2$ singularity.

 \end{example}
 
  \begin{remark} \label{Small} 
 Let $f:X\ra Z$, $L$ and $\tau$ be as in Theorem \ref{birational}. . Assume also that  $\dim E \le n-3$ (in particular $f$ is small).
 It follows  by \cite[Theorem 2.1(II.ii)]{And95} and \cite[Lemma 2.1]{BHN13} that $E$ is irreducible, it is contained in the non-Gorenstein locus of $X$,  is  contracted to a point and $(E,L_{|E})=(\bb P^{n-3}, \m O(1))$. 
 \end{remark}

\end{document}